\documentclass[a4paper,reqno,12pt]{amsart}
\usepackage{amsmath,amssymb, amsthm}
\usepackage{graphicx, enumerate, enumitem, color}
\usepackage[usenames,dvipsnames,svgnames,table]{xcolor}
\usepackage{xparse}
\usepackage{tikz}
\usetikzlibrary{decorations.markings,arrows}
\usepackage[utf8]{inputenc}

\colorlet{mylinkcolor}{violet}
\colorlet{mycitecolor}{YellowOrange}
\colorlet{myurlcolor}{Aquamarine}
\usepackage{hyperref}
\usepackage{xcolor}
\hypersetup{
    pdftitle={The excluded minors for isometric realizability in the plane},
    pdfauthor={},
    colorlinks,
    linkcolor={red!50!black},
    citecolor={blue!50!black},
    urlcolor={blue!80!black}
}

\makeatletter
\newtheorem*{rep@theorem}{\rep@title}
\newcommand{\newreptheorem}[2]{%
\newenvironment{rep#1}[1]{%
 \def\rep@title{#2 \ref{##1}}%
 \begin{rep@theorem}}%
 {\end{rep@theorem}}}
\makeatother

\newreptheorem{theorem}{Theorem}
\newreptheorem{corollary}{Corollary}

\newtheorem{theorem}{Theorem}

\newtheorem{question}[theorem]{Question}
\newtheorem{corollary}[theorem]{Corollary}

\newtheorem{proposition}[theorem]{Proposition}

\newtheorem{claim}[theorem]{Claim}

\newtheorem{lemma}[theorem]{Lemma}
\theoremstyle{remark}

\newcommand{\R}{\mathbb{R}}

\let\le\leqslant
\let\ge\geqslant
\let\leq\leqslant
\let\geq\geqslant

\DeclareMathOperator{\degree}{deg}
\DeclareMathOperator{\ex}{ex}

\def\be{\begin{equation}}
\def\ee{\end{equation}}

\makeatletter

\makeatletter

\tikzset{->-/.style={decoration={
  markings,
  mark=at position #1 with {\arrow[scale=1.5]{latex}}},postaction={decorate}}}

\newcommand{\definenodes}{%
\tikzstyle{vtx}=[circle,draw,thick,fill=black!10]
\node[vtx] (0) at (0,0) {\tiny $0$};
\node[vtx] (1) at (0,2) {\tiny $1$};
\node[vtx] (2) at (-0.5,1) {\tiny $2$};
\node[vtx] (3) at (-1.25,1) {\tiny $3$};
\node[vtx] (4) at (0.5,1) {\tiny $4$};
\node[vtx] (5) at (1.25,1) {\tiny $5$};%
}

\newcommand{\definenodesleft}{%
\tikzstyle{vtx}=[circle,draw,thick,fill=black!10]
\node[vtx] (0) at (0,0) {\tiny $0$};
\node[vtx] (1) at (0,2) {\tiny $1$};
\node[vtx] (2) at (-0.5,1) {\tiny $2$};
\node[vtx] (3) at (-1.25,1) {\tiny $3$};
}

\newcommand{\definenodesright}{%
\tikzstyle{vtx}=[circle,draw,thick,fill=black!10]
\node[vtx] (0) at (0,0) {\tiny $0$};
\node[vtx] (1) at (0,2) {\tiny $1$};
\node[vtx] (4) at (0.5,1) {\tiny $4$};
\node[vtx] (5) at (1.25,1) {\tiny $5$};%
}

\newcounter{axiomnumber}

\DeclareDocumentCommand\drawdirectedaxiom{ m g g o o o}{%
	\fbox{\begin{tikzpicture}[scale=1,inner sep=1.5pt]
	\definenodes
	\node at (-1.25,2) {\footnotesize A\arabic{axiomnumber}};
	\draw[thick,->-=1.0] #1;
	\IfNoValueF{#2}{%
		\draw[thick,->-=1.0] #2;%
		\IfNoValueF{#3}{%
			\draw[thick,->-=1.0] #3;%
		}%
	}%
	\IfNoValueF{#4}{%
		\draw[thick,dotted] #4;%
		\IfNoValueF{#5}{%
			\draw[thick,dotted] #5;%
			\IfNoValueF{#6}{%
				\draw[thick,dotted] #6;%
			}%
		}%
	}%
	\end{tikzpicture}}%
	\stepcounter{axiomnumber}%
}

\DeclareDocumentCommand\drawcaseleft{ m m g }{%
	\begin{tikzpicture}[scale=.75,inner sep=1.5pt]
	\definenodesleft
	\draw[very thick] #1;
	\draw[very thick] #2;
	\IfNoValueF{#3}{%
		\draw[very thick] #3;%
	}%
	\end{tikzpicture}%
}

\DeclareDocumentCommand\drawcaseright{ m m g }{%
	\begin{tikzpicture}[scale=.75,inner sep=1.5pt]
	\definenodesright
	\draw[very thick] #1;
	\draw[very thick] #2;
	\IfNoValueF{#3}{%
		\draw[very thick] #3;%
	}%
	\end{tikzpicture}%
}

\DeclareDocumentCommand\writeproof{ s o o o g g g }{%
	\tiny
	\begin{minipage}[b]{4.3em}
		\IfBooleanTF{#1}{%
	        \IfNoValueF{#2}{$\Delta(#2)$\\}%
	        \IfNoValueF{#3}{$\Delta(#3)$\\}%
	        \IfNoValueF{#4}{$\Delta(#4)$\\}%
	        \IfNoValueF{#5}{#5\\}%
	        \IfNoValueF{#6}{#6\\}%
	        \IfNoValueF{#7}{#7\\}%
	        (in $T_2$)\\
	   }{%
	        \IfNoValueF{#2}{$\Delta(#2)$\\}%
	        \IfNoValueF{#3}{$\Delta(#3)$\\}%
	        \IfNoValueF{#4}{$\Delta(#4)$\\}%
	        \IfNoValueF{#5}{#5\\}%
	        \IfNoValueF{#6}{#6\\}%
	        \IfNoValueF{#7}{#7}%
	   }%
	\end{minipage}
}

\newcommand{\Lpspace}[2]{\ell_{#1}^{#2}}

\title[The excluded minors for isometric realizability in the plane]{The excluded minors for isometric realizability in the plane}

\makeatletter
\let\old@setaddresses\@setaddresses
\def\@setaddresses{\bigskip\bgroup\parindent 0pt\let\scshape\relax\old@setaddresses\egroup}
\makeatother

\author[S.~Fiorini]{Samuel Fiorini}
\author[T.~Huynh]{Tony Huynh}
\address[S.~Fiorini, T.~Huynh]{Mathematics Department \\
  Universit\'e Libre de Bruxelles\\
  Brussels\\
  Belgium}
\email{sfiorini@ulb.ac.be, tony.bourbaki@gmail.com}

\author[G.~Joret]{Gwena\"{e}l Joret}
\address[G.~Joret]{Computer Science Department \\
  Universit\'e Libre de Bruxelles\\
  Brussels\\
  Belgium}
\email{gjoret@ulb.ac.be}

\author[A.~Varvitsiotis]{Antonios Varvitsiotis}
\address[A.~Varvitsiotis]{
School of Physical and Mathematical Sciences, Nanyang Technological University, Singapore \& 
 Centre for Quantum Technologies,
 National University of Singapore \\
 Singapore}
\email{avarvits@gmail.com}

\begin{document}
\begin{abstract}
Let $G$ be a graph and $p \in [1, \infty]$. The parameter  $f_p(G)$  is the least integer $k$ such that for all  $m$ and all vectors $(r_v)_{v \in V(G)} \subseteq \mathbb{R}^m$, there exist vectors $(q_v)_{v \in V(G)} \subseteq \mathbb{R}^k$ satisfying
$$\|r_v-r_w\|_p=\|q_v-q_w\|_p, \ \text{ for all }\ vw\in E(G).$$
It is easy to check that $f_p(G)$ is always finite and that it is minor monotone.  By the graph minor theorem of Robertson and Seymour~\cite{RS04}, there are a finite number of excluded minors for the property $f_p(G) \leq k$.   

  In this paper, we determine the complete set of excluded minors for $f_\infty(G) \leq 2$.  The two excluded minors are the wheel on $5$ vertices and the graph obtained by gluing two copies of $K_4$ along an edge and then deleting that edge.
  We also show that the same two graphs are the complete set of excluded minors for $f_1(G) \leq 2$.  In addition, we give a family of examples that show that $f_\infty$ is unbounded on the class of planar graphs and $f_\infty$ is not bounded as a function of tree-width.
\end{abstract}

\maketitle
\section{Introduction}
Let $X$ be a finite set and $d: X\times X \rightarrow  \mathbb{R}_{\ge0}$. We say that $(X,d)$ is a {\em metric space}
if  $d$  satisfies the following properties: $(i)$ 
 $d(i,j)=d(j,i)$ for all $i,j\in X$, $(ii)$  $ d(i,j)=0$ if and only if  $i=j$, and  $ (iii)$ $d(i,j)\le d(i,k)+d(k,j)$ for all 
 $i,j,k\in X$.  
 For $x \in \mathbb{R}^m$ define $\|x\|_p:=(\sum_{i=1}^m|x_i|^p)^{1/p}$  and $\|x\|_\infty:=\max_{i=1}^m |x_i|$.  Recall that $\|\cdot\|_p$  is a norm for all $p \in [1, \infty]$. Throughout this article we denote by $\ell_p^m$ the metric space  $(\mathbb{R}^m,d_p)$ where $d_p(x,y)=\|x-y\|_p$.
 
A natural way for comparing  two metric spaces  $(X,d)$ and $(X',d')$  is through the use of distance preserving maps from one space to the other. Formally, an  {\em isometric embedding}  of $(X,d)$ into $(X',d')$ is a function $\phi : X\rightarrow X'$ such that $d(x,y)=d'(\phi(x),\phi(y))$ for all $x,y\in X$.   
 
Typically, the requirement that all pairwise distances are preserved exactly is too restrictive to be useful  in practice. To cope with this, a  successful theory of embeddings with distortion has been developed, where   the requirement that distances are preserved  exactly is relaxed to the requirement that no distance shrinks or stretches excessively. In this direction, the celebrated theorem of Bourgain~\cite{Bourgain85}
asserts that every $n$-point metric space can be embedded  into an $\ell_p^{O(\log^2 n)}$ space with $O(\log n)$ distortion. Moreover, this is best possible up to a constant factor.  

Another popular approach is to only require a \emph{subset} of the distances to be preserved exactly.  This viewpoint is very graph theoretical, and is the approach that we take in this paper. 

All graphs in this paper are finite and do not contain loops or parallel edges.  A graph $H$ is a \emph{minor} of a graph $G$, if $H$ can be obtained from a subgraph of $G$ by contracting some edges.  
When taking minors, we always suppress parallel edges and loops.

Let $G$ be a graph and $p \in [1, \infty]$. We define  $f_p(G)$ to  be the least integer $k$ such that for all $m$ and all vectors $(r_v)_{v \in V(G)} \subseteq \mathbb{R}^m$, there exist vectors $(q_v)_{v \in V(G)} \subseteq \mathbb{R}^k$ satisfying
$$\|r_v-r_w\|_p=\|q_v-q_w\|_p, \ \text{ for all }\ vw\in E(G).$$  

It is not obvious that this parameter is always finite, but from the conic version of Carath\'eodory's Theorem, it follows  that $f_p(G) \le \binom{n}{2} $ for all $p \in [1, \infty]$ and all $n$-vertex graphs $G$ (see \cite{B90} and \cite[Proposition 11.2.3]{DL97}). 
For $p=2$, Barvinok~\cite{Barvinok95} showed the better bound $f_2(G) \leq (\sqrt{8m+1} - 1)/2$ for graphs $G$ with $m$ edges. 

Let $K_n$ denote the complete graph on $n$ vertices. The study of $f_p(K_n)$ for varying values of $p\in [1,\infty]$ is a fundamental problem in the theory of metric embeddings.  For the case $p=\infty$, Holsztynski \cite{H78} (and subsequently Witsenhausen \cite{W86}) showed that 
$$\left\lfloor \frac{2n}{3} \right\rfloor \leq f_\infty (K_n) \leq n-2, \text{ for } n\ge 4.$$  
Furthermore, Witsenhausen \cite{W86} showed that 
$ f_1 (K_n) \geq n-2$  for $n\ge 3$, 
which was later improved to 
$$f_1 (K_n) \geq \binom{n-2}{2}, \text{ for } n\ge 3,$$ by Ball \cite{B90}. 
 Lastly, Ball \cite{B90} also showed that 
 $$f_p(K_n)\ge \binom{n-1}{2},  \text{ for all } 1<p<2 \text{ and } n\ge 3 $$ and that there is a constant $c$ such that 
 $$f_\infty(K_n) \geq n-cn^{3/4}, \text{ for all $n$.}$$ 

The lower bound of  $n-cn^{3/4}$  uses the \emph{biclique covering number}, which is the minimum number of complete bipartite subgraphs needed to cover the edges of a graph. Rödl and Ruciński \cite{RR97} have since shown that there is a constant $c$ such that for every $n$ there exists an $n$-vertex graph that cannot be covered with $n-c\log n$ complete bipartite subgraphs.  
This implies that there is a constant $c$ such that $$f_\infty(K_n) \geq n-c \log n, \text{ for all $n$.}$$

The parameters $f_p(G)$ are also widely studied in rigidity theory.  We refer the interested reader to Kitson~\cite{Kitson15} and Sitharam and Gao~\cite{SG10} and the references therein. 

It is easy to show that for all $p \in [1, \infty]$, the parameter $f_p(G)$ is minor monotone. By the graph minor theorem of Robertson and Seymour~\cite{RS04}, there are a finite number of minor-minimal graphs $G$ with 
$f_p(G) > k$. We call these graphs the \emph{excluded minors} for $f_p(G) \leq k$.

The excluded minors for $f_2(G) \leq 1$, $f_2(G) \leq 2$, and $f_2(G) \leq 3$ were determined by Belk and Connelly~\cite{Belk,BC07}. 

\begin{theorem}[\cite{Belk,BC07}] \label{belkconnelly}
For every graph $G$,
\begin{itemize}
\item[(i)] $f_2(G) \le 1$ iff $G$ has no $K_3$ minor; 
\item[(ii)] $f_2(G)\le 2$ iff $G$ has no $K_4$ minor;
\item[(iii)] $f_2(G)\le 3$ iff $G$ has no $K_5$ minor and no $K_{2,2,2}$ minor. 
\end{itemize}
\end{theorem}

In this article we mainly focus on the case $p=\infty$.
The $\ell_\infty$-spaces are particularly interesting due to their ``universal'' nature in terms of isometric embeddings, as illustrated by the following theorem of Fr\'echet.

\begin{theorem}[\cite{F10}] \label{Frechet}Every $n$-point metric space can be isometrically embedded in~$\Lpspace{\infty}{n-1}$.
\end{theorem}

Theorem \ref{Frechet} allows us to rephrase the condition $f_\infty(G) \leq k$ as follows. Let $G$ be a graph and $d: E(G)\rightarrow \mathbb{R}_{\ge 0}$. The {\em length} of a path $P$ in $G$ is defined as 
 $\sum_{e\in E(P) }d_e$. Throughout this work we call  $d: E(G)\rightarrow \mathbb{R}_{\ge 0}$ a {\em distance function on $G$} if for all edges $xy\in E(G)$, 
 every path from $x$ to $y$ has length at least $d_{xy}$ 
 (in other words, the path consisting of the edge $xy$ is a shortest path). 
 We remark that $d_{xy} = 0$ is allowed in this definition, 
 and that $d$ defines a corresponding metric space $X$ on at most $|V(G)|$ points as follows. First contract all edges $xy$ with $d_{xy} = 0$, 
 and then consider the shortest path lengths between pairs of vertices. 
 Hence, by Theorem \ref{Frechet}, $f_\infty(G) \leq k$ if and only if for
 all distance functions $d$ on $G$, there exist vectors $(q_v)_{v \in V(G)} \subseteq \mathbb{R}^k$ satisfying
$$\|q_x-q_y\|_\infty=d_{xy}, \ \text{ for all }\ xy\in E(G).$$

Note that for all $p,q \in [1, \infty]$, $\ell_p^1=\ell_q^1$. Thus, by Theorem \ref{belkconnelly}, 
$f_\infty(G) \leq 1$ if and only if $G$ has no $K_3$ minor.  In this paper we determine the complete
set of excluded minors for $f_\infty(G)\le 2$. Let $W_4$ denote the wheel on $5$ vertices and $K_4 +_e K_4$ be the graph obtained by gluing two copies of $K_4$ along an edge $e$ and then deleting $e$, see Figure~\ref{fig:excluded_minors}.  Using techniques from rigidity matroids, Sitharam and Willoughby~\cite{SW15} determined $f_\infty(G)$ for all graphs $G$ with at most 5 vertices, except for $W_4$. They conjectured that $W_4$ is an excluded minor for $f_\infty (G) \leq 2$, and that $W_4$ is the \emph{only} excluded minor for $f_\infty(G) \leq 2$. We verify their first conjecture, but disprove the second by showing that $K_4 +_e K_4$ is also an excluded minor for $f_\infty(G) \leq 2$.

The following is our main result.

\begin{theorem}[Main Theorem]
\label{main}
The excluded minors for $f_\infty (G) \leq 2$ are $W_4$  and $K_4 +_e K_4$. 
\end{theorem}

The proof of Theorem~\ref{main} is given in Section \ref{sec:main}. Note that unlike the $p=2$ case, given points $x,y, x',y' \in \mathbb{R}^m$ with $\|x-y\|_\infty=\|x'-y'\|_\infty$ there does not necessarily exist an isometry of $\Lpspace{\infty}{m}$ which maps $x$ to $x'$ and $y$ to $y'$.
For example, take $x=x'=(0,0)$ and $y=(0,1), y'=(1,1)$ in $\Lpspace{\infty}{2}$. Indeed, the isometries of $\ell_\infty^m$ correspond to signed permutation matrices. Therefore, our proof technique for the  $p=\infty$ case is  quite different from the  $p=2$ case.  For example, we will show that the property $f_\infty(G) \leq 2$ is not closed under taking $2$-sums.   

We also prove the following result, which follows from Theorem~\ref{main} with a little extra work. 

\begin{corollary} \label{cor:main}
The excluded minors for $f_1 (G) \leq 2$ are $W_4$  and $K_4 +_e K_4$. 
\end{corollary}

Robertson and Seymour~\cite{RS95} proved that testing for a fixed minor can be done in cubic time.  Therefore, our results give an explicit cubic-time algorithm to test if $f_1(G) \leq 2$ (equivalently $f_\infty(G) \leq 2)$.  We simply have to test if our input graph contains a $W_4$ minor or a $K_4 +_e K_4$ minor. 

In a previous version of this paper, we asked whether $f_\infty$ is bounded on the class of planar graphs.  We also asked whether $f_\infty$ is bounded as a function of tree-width. We now have found an example that shows that the answer to both of these questions is negative.

\begin{theorem} \label{thm:example}
For every $k$ there exists a planar graph $G$ with tree-width $3$ such that $f_{\infty}(G) \geqslant k$.
\end{theorem}

\subsubsection*{Paper Organization.} 
In Section \ref{sec:potentials} we present a few equivalent ways to think about $f_\infty(G)$ and prove some upper and lower bounds.  In Section~\ref{sec:k7}, we show $f_\infty(K_7)=5$. In Section~\ref{sec:bounds} we show that we can suppress degree-$2$ vertices when computing $f_\infty(G)$. 
In Section~\ref{sec:twographs} we show that $W_4$ and $K_4 +_e K_4$ are 
excluded minors for $f_{\infty}(G) \leq 2$.  In Section \ref{sec:main} we show that $W_4$ and $K_4 +_e K_4$ are the \emph{only} excluded minors for $f_{\infty}(G) \leq 2$, and explain how to deduce Corollary~\ref{cor:main} from the main theorem. We conclude the paper in Section \ref{sec:openproblems} by proving Theorem~\ref{thm:example} and discussing some open problems.

\begin{figure}
\centering
\begin{tikzpicture}[scale=1.25,inner sep=2pt]
\tikzstyle{vtx}=[circle,draw,thick,fill=black!10]
\node[vtx] (0) at (0,0) {};
\node[vtx] (1) at (0,2) {};
\node[vtx] (2) at (-0.5,1) {};
\node[vtx] (3) at (-1.25,1) {};
\node[vtx] (4) at (0.5,1) {};
\node[vtx] (5) at (1.25,1) {};
\node at (1,0) {$K_4 +_e K_4$};
\node[vtx] (a) at (-3.125,0.125) {};
\node[vtx] (b) at (-4.825,0.125) {};
\node[vtx] (c) at (-4.825,1.825) {};
\node[vtx] (d) at (-3.125,1.825) {};
\node[vtx] (e) at (-4,1) {};
\node at (-2.675,0) {$W_4$};
\draw[thick] (0)--(2) (1)--(2) (0)--(3) (1)--(3) (2)--(3) (0)--(4) (1)--(4) (0)--(5) (1)--(5) (4)--(5) (a)--(b)--(c)--(d)--(a) (e)--(a) (e)--(b) (e)--(c) (e)--(d);
\end{tikzpicture}
\caption{The excluded minors for $f_{\infty}(G) \leq 2$.}
\label{fig:excluded_minors}
\end{figure}
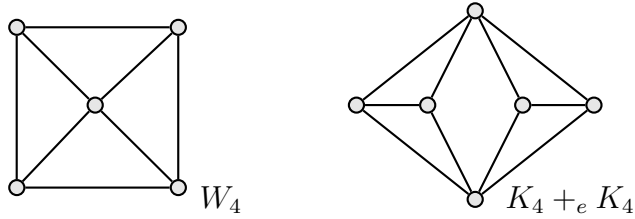

\section{Potentials and Implicit Realizations} \label{sec:potentials}

In this section we present several equivalent ways to think about the parameter $f_\infty(G)$.  

Consider an $n$-vertex graph $G$, a distance function $d$ on $G$, and a realization of $(G,d)$ in $\Lpspace{\infty}{k}$; that is, a collection of points $(q_v)_{v \in V(G)} \in \mathbb{R}^k$ such that $|| q_v - q_w ||_\infty = d_{vw},$ for all $vw \in E(G)$. We can write a $k \times n$ matrix whose columns are the vectors $q_v$ for $v \in V(G)$. In this section we analyze this matrix by looking at its rows, which turn out to be potentials of a natural directed graph associated to $(G,d)$.

Let $D$ be an edge-weighted directed graph and let  $l : A(D) \to \mathbb{R}$ be the length function on the arcs of $D$. Note  that negative lengths are allowed. A function $p : V(D) \to \mathbb{R}$ is called a \emph{potential on  $D$} if $p(v) - p(u) \leqslant l(a),$ for all arcs $a = (u,v) \in A(D)$.  We recall  the following well-known result characterizing  the existence of a potential.

\begin{theorem} \label{thm:potential}
A weighted directed graph   $(D,l)$ admits a potential if and only if it does not contain any negative length directed cycle.
\end{theorem}

Now let $D = D(G,d)$ be the weighted directed graph obtained from $(G,d)$ as follows.  First, we bidirect all edges of $G$. For every edge $uv \in E(G)$, we define the length of both $(u,v)$ and $(v,u)$ to be $d_{uv}$. That is, the length function $l$ on $D$ is given by
\begin{equation}
\label{eq:lengths}
l(u,v) = l(v,u) := d_{uv}, \quad \forall uv \in E(G).
\end{equation}
Note that $p : V(D) \to \mathbb{R}$  is a potential on $D$ if and only if $|p(v)-p(u)|\le d_{uv},\ \forall uv\in E(G)$. An edge $uv \in E(G)$ is \emph{tight} for a potential $p$ on $D$ if  $|p(v) - p(u)|= d_{uv}$.

Let $(q_v)_{v \in V(G)}$ be a realization of $(G,d)$ in $\Lpspace{\infty}{k}$.
Clearly, if we define $p_i(v) := q_v(i)$ for $i \in [k]$ and $v \in V(G)$, we have that $p_i$ is a potential for all $i \in [k]$.  Moreover, every edge of $G$ is tight in some $p_i$. It is easy to see that the converse also holds.  
\begin{lemma} \label{lem:potentials}
Let $G$ be a graph. A distance function $d$ on $G$ admits a realization $(q_v)_{v \in V(G)}$ in $\ell_\infty^k$ if and only if the directed graph $D = D(G,d)$ with lengths as in \eqref{eq:lengths} admits a collection of potentials $(p_i)_{i \in [k]}$ such that every edge $uv \in E(G)$ is tight in some $p_i$. Moreover, in this equivalence we can take $q_v(i) = p_i(v),$ for all $i \in [k]$ and $v \in V(G)$.
\end{lemma}

In view of Lemma~\ref{lem:potentials}, we get  a combinatorial approach for  constructing and analyzing realizations. For $F \subseteq E(G)$, let $\overrightarrow{F}$ denote some orientation of $F$. We say that $\overrightarrow{F}$ is a   \emph{feasible orientation} (with respect to $d$) if there exists a potential $p$ on $D(G,d)$ such that $p(v) - p(u) = d_{uv},$ for all $(u,v) \in \overrightarrow{F}$. 
See Figure~\ref{fig:W4orientations} for an illustration. 
 We say that $F \subseteq E(G)$ is \emph{feasible} if it admits a feasible  orientation. If a set of edges is not feasible, we say that it is \emph{infeasible}. Notice that $\overrightarrow{F}$ is a feasible orientation if and only if the opposite orientation $\overleftarrow{F}$ is a feasible orientation.  Furthermore, note that a subset of a feasible set is also feasible.
 
\begin{figure}
\centering
\includegraphics[width=0.95\textwidth]{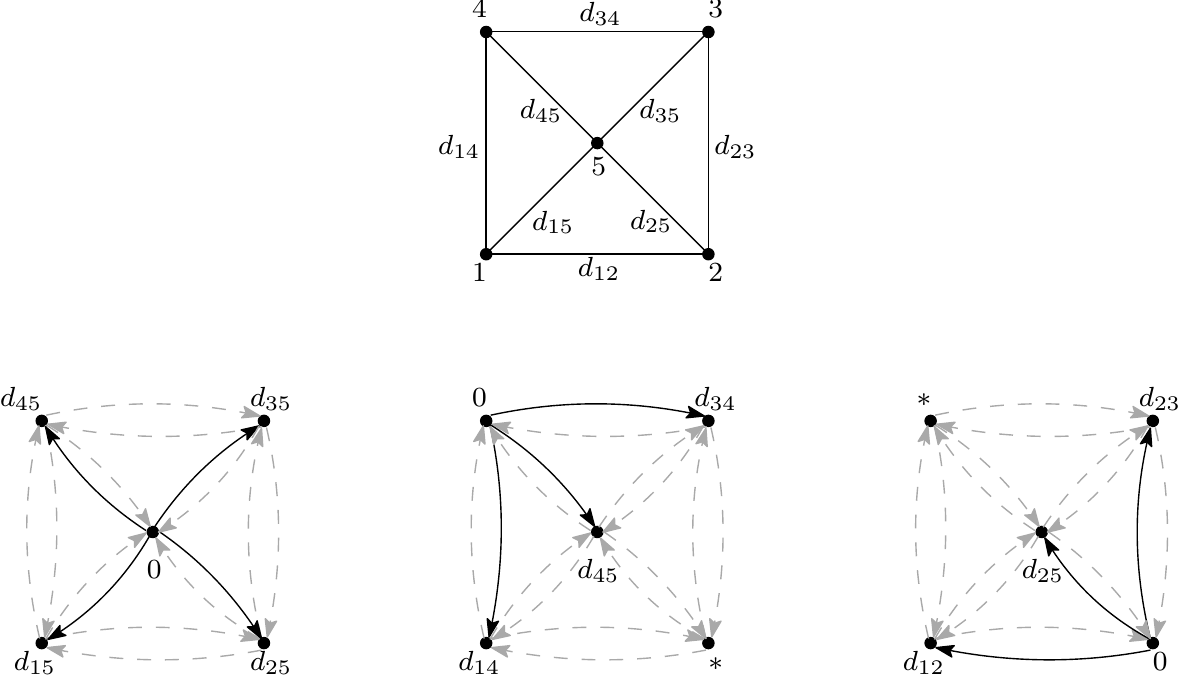}
\caption{If $G$ denotes the $5$-vertex wheel $W_4$, then $(G,d)$ admits a realization in $\ell_\infty^3$ for all distance functions $d$, as shown by these three potentials. (The two values labeled $*$ are not used to realize any edge so they can be set to any value that is feasible.)}
\label{fig:W4orientations}
\end{figure} 
 
 The  notion of feasible sets allows to reformulate Lemma \ref{lem:potentials} as follows. 

\begin{lemma} \label{lem:feasible}
Let $G$ be a graph and $d$ be a distance function on $G$. The pair $(G,d)$ admits a realization in $\ell_\infty^k$ if and only if there exist  feasible  sets $(F_i)_{i \in [k]}$ such that $\cup_{i=1}^k F_i=E(G)$.
\end{lemma}
Given an  orientation $\overrightarrow{F}$, we define   a  modification of  the length function $l(d)$  as follows.
\begin{equation}
\label{eq:lengths_with_forcing}
l(u,v) := 
\begin{cases}
d_{uv}, &\text{if } uv \in E(G), (u,v) \notin \overrightarrow{F};\\
-d_{uv}, &\text{if } (u,v) \in \overrightarrow{F}.
\end{cases}
\end{equation}
We denote this length function by  $l(d,\overrightarrow{F})$. Note that $\overrightarrow{F}$ is a feasible orientation  if and only if  $(G,l(d,\overrightarrow{F}))$ admits a potential. By Theorem~\ref{thm:potential}, this happens if and only if the weighted digraph $(G,l(d,\overrightarrow{F}))$
does not contain  a directed cycle of negative length.

We demonstrate the usefulness of Lemma \ref{lem:feasible} by quickly deriving some non-trivial upper and lower bounds for $f_\infty(G)$.  

Note that for every distance function $d$ on $G$ and every vertex $v$ of $G$, the star centered at $v$ is always feasible with respect to $d$, as can be seen by orienting all the edges of the star outwards (as in Figure~\ref{fig:W4orientations}). From this we obtain the following upper bound.

\begin{lemma} \label{vertexcover}
For every graph $G$, 
$$f_\infty(G)\le \tau(G),$$
where $\tau(G)$ denotes the minimum size of a vertex cover of $G$.
\end{lemma}

We say that a distance function $d$ is \emph{generic} with respect to $G$ if for every cycle $C$ in $G$ and $S \subseteq E(C)$, we have $\sum_{e \in S} d_{e} \neq \sum_{e \in E(C) \setminus S} d_e$. Every distance function $d$ on $G$ can be perturbed to a nearby generic distance function $d'$. Furthermore, we have $f_\infty(G) \leqslant k$ if and only if $(G,d)$ can be realized in $\ell_{\infty}^k$ for every \emph{generic} distance function $d$.

Observe that if $d$ is generic, every feasible set is acyclic. Therefore, we immediately obtain the following lemma.
 
\begin{lemma} \label{arboricity}
For every graph $G$, 
$$f_\infty(G)\ge \Upsilon(G),$$
where $\Upsilon(G)$ denotes the minimum number of forests required to partition $E(G)$.
\end{lemma}

Our next result implies that, if $d$ is generic, every maximal feasible set is a spanning~forest.

\begin{lemma} \label{lem:feasible_contains_tree}
Let $G$ be a graph and $d$ be a distance function on $G$.  Then every maximal feasible set $F \subseteq E(G)$ contains a spanning forest.
\end{lemma}

\begin{proof}
Towards a contradiction, suppose that $F \subseteq E(G)$ is a maximal feasible set that does not contain a spanning forest of $G$. Let $X$ be the vertex set of a component of $(V(G),F)$ such that $G$ contains at least one edge with exactly one end in $X$.  Let $p$ be any potential that makes all the edges of $F$ tight but no other edges. Let $\Delta$ be as large as possible with the property that $p':= p + \Delta \sum_{v \in X} e_v$ is a potential, where $e_v$ denotes the characteristic vector for the vertex $v$. Then the set of edges that are tight with respect to $p'$ is a proper superset of $F$, a contradiction.
\end{proof}

\section{$f_{\infty}(K_7) = 5$} \label{sec:k7}

Since $\lfloor \frac{2n}{3} \rfloor=n-2$ for $n \in \{4,5,6\}$, it follows that $f_\infty(K_3)=2, f_\infty(K_4)=2, f_\infty(K_5)=3$, and $f_\infty(K_6)=4$.  Thus, $n=7$ is the smallest value for which $f_\infty(K_n)$ is unknown.  In this section we show that $f_{\infty}(K_7) = 5$.  This result is not needed for our main theorem but may be of independent interest. 

\begin{proposition}
$f_{\infty}(K_7) = 5$.
\end{proposition}
\begin{proof}
We already know that $f_{\infty}(K_7) \leq 5$, let 
us prove that $f_{\infty}(K_7) \geq 5$. 
To this aim, enumerate the vertices of $K_7$ as 
$v_1, \dots, v_7$, and define a linear ordering $L$ on 
its edges by letting, for $i < j$ and $k < \ell$, 
$$v_iv_j >_L v_kv_{\ell}$$ if $i < k$, or $i=k$ and 
$j < \ell$. 
Let $m:= 21$ be the number of edges. 
Define a distance function $d$ on the graph by letting 
$d(e) := 2^m + 2^r$ for each edge $e$, where $r$ is the rank of $e$ in the ordering $L$. (Thus $v_1v_2$ has rank $m$ and $v_6v_7$ has rank $1$.) It is easy to check that $d$ is a generic distance function. 

We claim that $(K_7, d)$ cannot be realized 
in $\Lpspace{\infty}{4}$. 
Arguing by contradiction, assume it can. 
Consider a partition of the edges into four 
feasible forests $F_1, \dots, F_4$. 
Before analyzing these, 
let us note a few properties of a feasible forest $F$ (the easy proofs are left to the reader).
\begin{enumerate}
\item a feasible orientation $\overrightarrow{F}$ of $F$ 
cannot contain a length-$2$ directed path, hence 
 $\overrightarrow{F}$ is uniquely determined 
 (up to reversing all arcs);
\item \label{prop:atmostone}
if $i < j < k < \ell$ then at most one 
of the two edges $v_iv_j$ and $v_kv_\ell$ is in $F$; 
\item \label{prop:atmostwo}
if $i < j < k < \ell$ then at most two 
of the three edges $v_iv_k$, $v_jv_k$, $v_jv_{\ell}$ 
are in $F$. 
\end{enumerate}

Now color each edge $e$ of the graph with the index $i$ of the 
forest $F_i$ it is included in. 
By~\eqref{prop:atmostone} 
we may assume without loss of generality that 
$v_1v_2$, $v_3v_4$, and $v_5v_6$ are colored $1$, $2$, and $3$ respectively.  
By the same property, none of the two edges $v_5v_7$, $v_6v_7$ are colored $1$ or $2$, 
and they cannot both be colored $3$ (otherwise $v_5v_6v_7$ would be a triangle in $F_3$), 
thus there exists $a\in\{5, 6\}$ such that $v_av_7$ is colored $4$. 

Next consider the four edges between the set $\{v_1,v_2\}$ and $\{v_3, v_4\}$.  
None of these is colored $3$ by~\eqref{prop:atmostone} (because of the edge $v_5v_6$) 
or $4$ (because of the edge $v_av_7$), so each of them is colored $1$ or $2$. 
Moreover, in order to avoid monochromatic triangles, the four edges are split into 
two matchings $M_1$ and $M_2$ of size $2$, colored $1$ and $2$ respectively. 

Let $X$ be the set of edges $v_iv_j$ with $i, j \geq 3$ that are distinct from $v_3v_4$. 
(Thus $|X| = 9$.)
No edge in $X$ is colored $1$ (because of $v_1v_2$). 
We claim that no edge in $X$ is colored $2$ either. 
This is clear for those not incident to $v_3$, thanks to the edge of $M_2$ that is incident to $v_3$. 
Now, suppose for a contradiction that $f\in X$ is incident to $v_3$ and is colored $2$. 
Then letting $e$ be the edge of $M_2$ incident to $v_4$, 
we see that the edges $e, v_3v_4, f$ are all in $F_2$, contradicting property~\eqref{prop:atmostwo}. 

All edges in $X$ are colored $3$ or $4$ 
but $X$ has size $9$ and spans only $5$ vertices. 
Therefore, there is a monochromatic cycle in $X$. 
This final contradiction concludes the proof. 
\end{proof}

\section{Degree-$2$ Vertices} \label{sec:bounds}

In this section we show that we can essentially ignore degree-$2$ vertices when computing~$f_\infty(G)$.

Let $G_1$ and $G_2$ be graphs that each contain a clique $K$ of size $k$.  A \emph{$k$-sum} of $G_1$ and $G_2$ along $K$ is a graph obtained by 
gluing $G_1$ and $G_2$ along $K$ and then deleting some of the edges of $K$.  In the special case of $2$-sums, we use the notation $G_1 \oplus_e G_2$ if
we keep the edge $e$, and $G_1 +_e G_2$ if we delete the edge $e$.

\begin{lemma} \label{degree2}
Let $H$ be a graph and let $e\in E(H)$. If $f_\infty (H) \geq 2$, then $f_\infty (H)=f_\infty (H\oplus_e K_3)$.  
\end{lemma}

\begin{proof}
Set $G:=H\oplus_e K_3$, let $e=uv$ and let $w$ be the newly added vertex in $G$.  Clearly $f_\infty(G) \geq f_\infty (H)$ so it suffices to show that $f_\infty(G) \leq f_\infty (H)$. Let $d$ be any distance function on $G$. The restriction of $d$ to $H$ is also a distance function. Let $(F_i)_{i \in [k]}$ be a collection of $k := f_\infty(H) \geq 2$ feasible  sets of $(H,d)$ such that   $\cup_{i=1}^kF_i=E(H)$.

First, note that each $F_i$ is feasible in $G$. Indeed, since $d$ is a distance function, and in particular $d_{uw} + d_{wv} \geq d_{uv}$, we can extend any potential on $D(H,d)$ to a potential on $D(G,d)$ by carefully choosing the potential value at $w$ between the value at $u$ and that at $v$. 
Without loss of generality, we may assume that $uv \in F_1$. Now extend $F_2$ to a maximal feasible set $F'_2 \subseteq E(G)$. By Lemma~\ref{lem:feasible_contains_tree}, $F'_2$ contains a spanning forest. Hence, $F'_2$ contains either $wu$ or $wv$. Without loss of generality, assume that $wu \in F'_2$. 

Now let $\overrightarrow{F_1}$ be a feasible orientation of $F_1$. By reversing all the arcs of $\overrightarrow{F_1}$ if necessary, we may assume that $(u,v) \in \overrightarrow{F_1}$. We claim that $\overrightarrow{F'_1} := \overrightarrow{F_1} \cup \{(w,v)\}$ is a feasible orientation. Indeed, let $C$ be a negative directed cycle in $D = D(G,d)$ with respect to $l := l(d,\overrightarrow{F'_1})$. Since $\overrightarrow{F_1}$ is a feasible orientation, we may assume that $(u,w), (w,v) \in A(C)$.  
Now $l(u,w) + l(w,v) = d_{uw} - d_{wv} \geq -d_{uv} = l(u,v)$, which means that the length of $C$ does not increase if we shortcut it from $u$ to $v$. Since $\overrightarrow{F_1}$ is a feasible orientation, the length of the shortcut cycle is nonnegative, which contradicts our assumption that $C$ has negative length. Hence, $\overrightarrow{F'_1}$ is a feasible orientation and the corresponding edge set $F'_1$ is feasible.

We have found $k$ feasible sets $F'_1$, $F'_2$, $F_3$, \ldots, $F_k$ that cover each edge of $G$. Thus $(G,d)$ can be realized in $\Lpspace{\infty}{k}$. The lemma follows.
\end{proof}

We note that the assumption that  $f_\infty (H) \geq 2$ in Lemma \ref{degree2}  is necessary.  This can easily be seen by taking $H=K_2$ and $G=K_3$.  

We say that $G$ is obtained from $H$ by \emph{subdividing an edge $e$} if $G=H +_e K_3$.

\begin{lemma} \label{subdivision}
Let $G$ and $H$ be graphs such that $G$ is obtained from $H$ by subdividing an edge.  Then $f_\infty(G)=f_\infty (H)$.
\end{lemma}

\begin{proof}
Clearly $f_\infty(G) \geq f_\infty (H)$  since $H$ is a minor of $G$. It remains to prove $f_\infty(G) \leq f_\infty (H)$.
If $f_\infty (H)=1$ then $H$ is a forest, and so is $G$, implying $f_\infty (G)=1$. 
Hence we may assume that $f_\infty (H) \geq 2$. 
Say that $G$ is obtained from $H$ by subdividing an edge $uv$ with a new vertex $w$. 
Let $G' := G + uv$. 
Since $G'$ is obtained from $H$ by adding a new vertex $w$ adjacent to the ends of the edge $uv$, 
we have that $f_\infty(G') = f_\infty (H)$ by Lemma~\ref{degree2}. 
The graph $G$ being a minor of $G'$, it follows that $f_\infty(G) \leq f_\infty(G') = f_\infty (H)$.
\end{proof}

\section{The graphs $W_4$ and $K_4 +_e K_4$} \label{sec:twographs}

In this section we show that $W_4$ and $K_4 +_e K_4$ are  
excluded minors for  $f_\infty (G) \leq 2$. 
\begin{figure}[b]
\centering
\begin{tikzpicture}[scale=1.5,inner sep=1.5pt]
\tikzstyle{vtx}=[circle,draw,thick,fill=black!10]
\node[vtx] (1) at (0,0) {\tiny $1$};
\node[vtx] (2) at (2,0) {\tiny $2$};
\node[vtx] (3) at (2,2) {\tiny $3$};
\node[vtx] (4) at (0,2) {\tiny $4$};
\node[vtx] (5) at (1,1) {\tiny $5$};

\draw[thick] (1) -- node[fill=white,inner sep=1pt,midway]{\tiny $18$} (2);
\draw[thick] (2) -- node[fill=white,inner sep=1pt,midway]{\tiny $17$} (3);
\draw[thick] (3) -- node[fill=white,inner sep=1pt,midway]{\tiny $20$} (4);
\draw[thick] (4) -- node[fill=white,inner sep=1pt,midway]{\tiny $24$} (1);
\draw[thick] (1) -- node[fill=white,inner sep=1pt,midway]{\tiny $200$} (5);
\draw[thick] (2) -- node[fill=white,inner sep=1pt,midway]{\tiny $200$} (5);
\draw[thick] (3) -- node[fill=white,inner sep=1pt,midway]{\tiny $200$} (5);
\draw[thick] (4) -- node[fill=white,inner sep=1pt,midway]{\tiny $200$} (5);
\end{tikzpicture}
\caption{$W_4$ and a distance function that cannot be realized in $\Lpspace{\infty}{2}$.}
\label{fig:bad_metric_W4}
\end{figure}
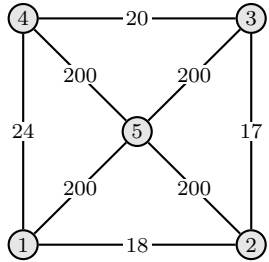

\begin{lemma} \label{fw4} We have that 
$f_\infty (W_4)=3$.
\end{lemma}

\begin{proof}
By Lemma \ref{vertexcover}, $f_\infty (W_4) \leq 3$.  Towards a contradiction suppose $f_\infty (W_4) \leq 2$. Let $d$ be the distance function on $W_4$ given in Figure~\ref{fig:bad_metric_W4} and let $q_1, \dots, q_5$ be an isometric embedding of $(G, d)$ in $ \Lpspace{\infty}{2}$.  Note that $q_1, \dots, q_4$ all lie on two consecutive sides of a square centered at $q_5$ with side length $400$.  By symmetry we may assume that $q_5=(200, -200)$, that $q_1=(x, 0)$ where $0 \leq x \leq 200$ and that $q_i(1) = 0$ or $q_i(2) = 0$ for $i \in \{2,3,4\}$.  We say that $(a,0)$ is  \emph{directly right of} $(b,0)$ if $b<a$ (in this case $(b,0)$ is \emph{directly left of} $(a,0)$),  $(0,c)$ is \emph{directly below} $(0,d)$ if $c <d$, and that $(a,0)$ and $(0,c)$ are \emph{diagonal}.   

We first consider the case that $q_4$ is directly right of $q_1$.  This implies that $q_3$ must be directly left of $q_4$ as $q_2$ would be too far from $q_1$ (if $q_3$ is directly right of $q_4$) or $q_3$ would be too far from $q_4$ (if $q_3$ and $q_4$ are diagonal).  Now, $q_2$ cannot be directly right of $q_3$ as $q_2$ would be too far from $q_1$, and $q_2$ cannot be directly left of $q_3$ as $q_2$ would be too close to $q_1$.  Thus, $q_2$ and $q_3$ are diagonal.  But now $\|q_1 - q_2\|_\infty \leq \|q_3 - q_2\|_\infty$, which is a contradiction.

We next consider the case that $q_4$ is directly left of $q_1$. Again, $q_3$ cannot be directly left of $q_4$.  Suppose that $q_3$ is directly right of $q_4$.  Again, $q_2$ cannot be directly right of or left of $q_3$. Thus, $q_2$ and $q_3$ are diagonal.  But now $\|q_2 - q_3\|_\infty \geq 20$, which is a contradiction.  Thus, $q_3$ and $q_4$ must be diagonal.   If $q_2$ is directly above or  directly below $q_3$, then $\|q_2 - q_1\|_\infty \geq 24$, which is a contradiction.  Thus, $q_2$ and $q_3$ are diagonal.  Since $d_{3,4}=20$, we must have $q_3=(-20, 0)$ or $q_4=(0,20)$.  In the first case, $\|q_2 - q_3\|_\infty \geq 20$ and in the second case $\|q_2 - q_1\|_\infty \geq 27$, both of which are contradictions. 

The remaining case is if $q_1$ and $q_4$ are diagonal.  Thus, $q_1=(24,0)$ or $q_4=(0,-24)$.  Suppose $q_1=(24,0)$.  If $q_2$ and $q_1$ are diagonal, then $\|q_1 - q_2\|_\infty \geq 24$, a contradiction.  If $q_2$ is directly right of $q_1$, then $q_3$ is too far away from $q_4$.  Thus, $q_2=(6,0)$.  Evidently, $q_3$ cannot be directly left of $q_2$.  If $q_3$ is directly right of $q_2$ we have $\|q_3 - q_4\|_\infty \geq 23$, a contradiction. If $q_3$ and $q_2$ are diagonal, then $q_3$ and $q_4$ are too close.  We finish with the subcase that $q_4=(0,-24)$.  Again, we must have $q_3=(0,-4)$.  If $q_2$ is directly below $q_3$, then $\|q_2 - q_1\|_\infty \geq 21$, a contradiction.  If $q_2$ and $q_3$ are diagonal, then $q_2 = (17,0)$ and is too close to $q_1$. This completes the subcase and the proof.
\end{proof}

\begin{lemma} \label{w4excluded}
The graph $W_4$ is an excluded minor for $f_{\infty}(G) \leq 2$. Moreover, $W_4$ is the only excluded minor for $f_{\infty}(G) \leq 2$ among all graphs with at most 5 vertices.  
\end{lemma}

\begin{proof}
By the previous lemma, $f_\infty(W_4) = 3$, so to prove that $W_4$ is an excluded minor it suffices to show that every proper minor $H$ of $W_4$ satisfies $f_\infty(H) \leq 2$. 
If $|V(H)| \leq 4$, then $f_\infty(H) \leq 2$ since $f_\infty(K_4) \leq 2$. Now, say $H$ is obtained from $W_4$ by only deleting edges. Deleting an edge yields a degree-$2$ vertex, which we can suppress by either Lemma \ref{subdivision} or Lemma \ref{degree2}.  Again, we get a graph with at most 4 vertices, so we are done.  

For the second part, let $H$ be an excluded minor for $f_\infty(G) \leq 2$  with $|V(H)|\le 5$. If $H$ has a $W_4$ minor, then $H=W_4$.  So we may assume that $H$ has no $W_4$ minor.  Let $e=ab$ and $f=ac$ be edges of $K_5$. 
By Lemma \ref{degree2} we have that  $f_{\infty}(K_5-\{e,f\})=f_{\infty}(K_4) = 2$. Since $H$ has no $W_4$ minor, this implies that $H$ is a minor of $K_5-\{e,f\}$.  But then, $f_\infty (H) \leq f_{\infty}(K_5-\{e,f\})=2$, which is a contradiction. 
\end{proof}

\begin{figure}[t]
\centering
\begin{tikzpicture}[scale=1.75,inner sep=1.5pt]
\definenodes
\draw[thick] (0) -- node[fill=white,inner sep=1pt,midway]{\tiny $71$} (2);
\draw[thick] (1) -- node[fill=white,inner sep=1pt,midway]{\tiny $53$} (2);
\draw[thick] (0) -- node[fill=white,inner sep=1pt,midway]{\tiny $77$} (3);
\draw[thick] (1) -- node[fill=white,inner sep=1pt,midway]{\tiny $88$} (3);
\draw[thick] (2) -- node[fill=white,inner sep=1pt,midway]{\tiny $78$} (3);
\draw[thick] (0) -- node[fill=white,inner sep=1pt,midway]{\tiny $74$} (4);
\draw[thick] (1) -- node[fill=white,inner sep=1pt,midway]{\tiny $79$} (4);
\draw[thick] (0) -- node[fill=white,inner sep=1pt,midway]{\tiny $46$} (5);
\draw[thick] (1) -- node[fill=white,inner sep=1pt,midway]{\tiny $36$} (5);
\draw[thick] (4) -- node[fill=white,inner sep=1pt,midway]{\tiny $79$} (5);
\end{tikzpicture}
\caption{$K_4 +_e K_4$ and a distance function  that cannot be realized in $ \Lpspace{\infty}{2}$.}
\label{fig:bad_metric}
\end{figure}
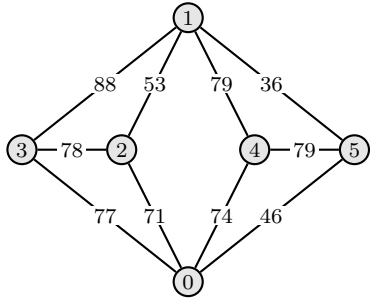

\begin{lemma}\label{lem:k4+k4}
We have that 
$f_\infty (K_4 +_e K_4)=3$.
\end{lemma}

\begin{proof}
To simplify notation, throughout this proof we  set  $G:=K_4 +_e K_4$. Furthermore, we  use the labeling  of the nodes of $G$ given in Figure  \ref{fig:bad_metric}. 

We first show that $f_\infty(G) \leq 3$.  Let $d$ be an arbitrary distance function on $G$.  Note that $F_0=\{02,03,04,05\}$ and $F_1=\{12,13,14,15\}$ are feasible sets because they are stars.  Thus, if $\{23,45\}$ is feasible, then $(G,d)$ can be realized in $\ell_{\infty}^3$ by Lemma \ref{lem:feasible}.
To conclude the proof  assume that $\{23,45\}$ is not feasible. Note that $F_3=\{30,31,32\}$ and $F_5=\{50, 51,54\}$ are feasible because they are stars. Let $F_3'$ and $F_5'$ be maximal feasible sets containing $F_3$ and $F_5$ respectively.  By Lemma \ref{lem:feasible_contains_tree}, $F_3'$ and $F_5'$ each span all the vertices of $G$.  Therefore, since $\{23,45\}$ is not feasible, we must have $\{02,12\} \cap F_5' \neq \emptyset$ and $\{04,14\} \cap F_3' \neq \emptyset$.  Let $F:=E(G) \setminus (F_3' \cup F_5')$.  Thus, $F$ is a subset of $\{02,04\}, \{12,14\}, \{02,14\}$, or $\{12,04\}$.  In the first two cases, $F$ is feasible since it is a subset of a star.  In the last two cases, note $\{(0,2), (4,1)\}$ and $\{(1,2), (4,0)\}$ are feasible orientations of $\{02,14\}$ and $\{12,04\}$, respectively. Hence, $F$ is also feasible in the last two cases. Since $F_3'$, $F_5'$ and $F$ are feasible sets covering all the edges of $G$, Lemma~\ref{lem:feasible} yields $f_\infty(G) \leq 3$. 

To  show that $f_\infty(G)=3$ it remains to exhibit a distance function  $d$ on $G$ such that $(G,d)$ is not realizable in $\ell_\infty^2$.  We exhibit such a distance function in Figure \ref{fig:bad_metric}.  Towards a contradiction, suppose that $E(G)$ can be partitioned into two feasible sets $T_1$ and $T_2$.  It is easy to check that $d$ is a generic distance function, and so $T_1$ and $T_2$ are both forests.\footnote{If one does not want to check genericity, simply perturb $d$ to a nearby generic distance function.} Thus, $|T_1|, |T_2| \leq |V(G)|-1=5$ edges.  Since $|E(G)|=10$, we conclude that $T_1$ and $T_2$ are both spanning trees. Let $T_L$ and $T_R$ be the subgraphs of $T_1$ induced by $\{0,1,2,3\}$ and $\{0,1,4,5\}$, respectively. By interchanging $T_1$ and $T_2$, we may assume that $|E(T_L)|=3$.  Therefore, there are six possibilities for each of $T_L$ and $T_R$, and these are shown in Figure \ref{prooftable}.  The six possibilities for $T_L$ are shown along the first column of the table, and the six possibilities for $T_R$ are shown along the first~row.  

We rule out each of the $36$ possibilities for $T_1$ by showing that at least one of $T_1$ or $T_2$ is infeasible. To do this, we show that for all orientations $\overrightarrow{T_1}$ and $\overrightarrow{T_2}$ of $T_1$ and $T_2$, at least one of $\overrightarrow{T_1}$ or $\overrightarrow{T_2}$ contains an infeasible orientation.  

If $abc$ forms a triangle in $G$, note that $\{(a,b), (b,c)\}$ is an infeasible orientation. Indeed, the triangle inequality combined with the fact that $d$ is generic imply that   the directed cycle $(a,b,c)$ is negative.  We denote this infeasible orientation as $\Delta(a,b,c)$.  In Figure~\ref{forbiddenconfig}, we list more infeasible orientations that do not come from triangles.  These infeasible orientations consist only of the oriented arcs in each picture.  However, for the benefit of the reader, we have included dashed edges to indicate the negative cycle in $D(G,d)$.  

The remainder of the proof is summarized in Figure \ref{prooftable}.  Each entry in the table gives the infeasible orientations to apply in order to obtain a contradiction.  For example, consider the fourth row of the table.  For this entire row, it suffices to only consider the edges in $E(T_L)$.  By symmetry, we may assume that $(0,2) \in \overrightarrow{T_L}$.  Next, $\Delta(3,0,2)$ implies that $(0,3) \in \overrightarrow{T_L}$.  Then, $A2$ implies $(1,3) \in \overrightarrow{T_L}$.  Since $(1,3), (0,2) \in \overrightarrow{T_L}$, we contradict $A1$.  Thus, $\Delta(3,0,2), A1$, and $A2$ are sufficient to derive a contradiction.  Sometimes the infeasible orientations need to be applied to $T_2$ instead of to $T_1$, in which case we have specified so.  
\end{proof}

\begin{figure}
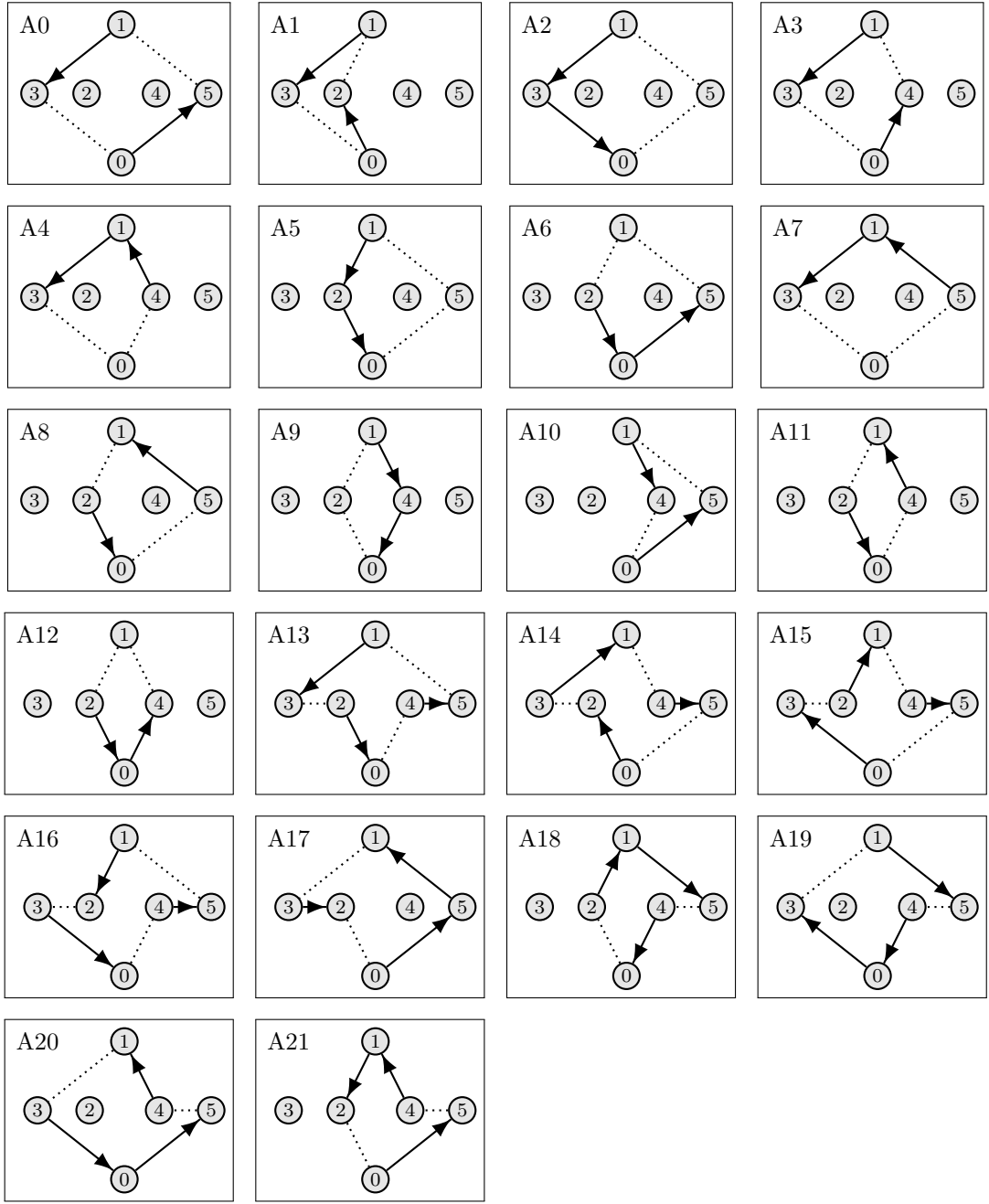

\centering
\begin{tabular}{c@{\hspace{.75em}}c@{\hspace{.75em}}c@{\hspace{.75em}}c}
\drawdirectedaxiom{(1)--(3)}{(0)--(5)}[(5)--(1)][(3)--(0)] 
&\drawdirectedaxiom{(1)--(3)}{(0)--(2)}[(3)--(0)][(2)--(1)] 
&\drawdirectedaxiom{(1)--(3)}{(3)--(0)}[(0)--(5)][(5)--(1)] 
&\drawdirectedaxiom{(1)--(3)}{(0)--(4)}[(3)--(0)][(4)--(1)]\\[1.5ex] 
\drawdirectedaxiom{(1)--(3)}{(4)--(1)}[(3)--(0)][(0)--(4)] 
&\drawdirectedaxiom{(1)--(2)}{(2)--(0)}[(0)--(5)][(5)--(1)] 
&\drawdirectedaxiom{(2)--(0)}{(0)--(5)}[(5)--(1)][(1)--(2)] 
&\drawdirectedaxiom{(5)--(1)}{(1)--(3)}[(3)--(0)][(0)--(5)]\\[1.5ex] 
\drawdirectedaxiom{(5)--(1)}{(2)--(0)}[(1)--(2)][(0)--(5)]
&\drawdirectedaxiom{(1)--(4)}{(4)--(0)}[(0)--(2)][(2)--(1)] 
&\drawdirectedaxiom{(1)--(4)}{(0)--(5)}[(4)--(0)][(5)--(1)] 
&\drawdirectedaxiom{(4)--(1)}{(2)--(0)}[(1)--(2)][(0)--(4)]\\[1.5ex] 
\drawdirectedaxiom{(2)--(0)}{(0)--(4)}[(4)--(1)][(1)--(2)] 
&\drawdirectedaxiom{(1)--(3)}{(2)--(0)}{(4)--(5)}[(3)--(2)][(0)--(4)][(5)--(1)] 
&\drawdirectedaxiom{(3)--(1)}{(0)--(2)}{(4)--(5)}[(2)--(3)][(1)--(4)][(5)--(0)] 
&\drawdirectedaxiom{(0)--(3)}{(2)--(1)}{(4)--(5)}[(1)--(4)][(5)--(0)][(3)--(2)]\\[1.5ex] 
\drawdirectedaxiom{(1)--(2)}{(3)--(0)}{(4)--(5)}[(2)--(3)][(0)--(4)][(5)--(1)] 
&\drawdirectedaxiom{(3)--(2)}{(0)--(5)}{(5)--(1)}[(1)--(3)][(2)--(0)] 
&\drawdirectedaxiom{(2)--(1)}{(1)--(5)}{(4)--(0)}[(5)--(4)][(0)--(2)] 
&\drawdirectedaxiom{(4)--(0)}{(0)--(3)}{(1)--(5)}[(3)--(1)][(5)--(4)]\\[1.5ex] 
\drawdirectedaxiom{(4)--(1)}{(3)--(0)}{(0)--(5)}[(1)--(3)][(5)--(4)] 
&\drawdirectedaxiom{(4)--(1)}{(1)--(2)}{(0)--(5)}[(2)--(0)][(5)--(4)] 
\end{tabular}
\caption{Infeasible Orientations A0--A21.}
\label{forbiddenconfig}
\end{figure}

\begin{figure} 
\begin{center}
\begin{tabular}{|c||c|c|c|c|c|c|}
\hline
&\drawcaseright{(1)--(5)}{(4)--(0)}
&\drawcaseright{(1)--(4)}{(5)--(0)}
&\drawcaseright{(4)--(5)}{(4)--(0)}
&\drawcaseright{(4)--(5)}{(4)--(1)}
&\drawcaseright{(4)--(5)}{(5)--(0)}
&\drawcaseright{(4)--(5)}{(5)--(1)}\\\hline\hline
\drawcaseleft{(1)--(3)}{(3)--(2)}{(2)--(0)}
&\writeproof*[0,5,4][5,4,1]{A20, A21}{A15}
&\writeproof*[0,4,5][4,5,1]{A7, A19}{A16}
&\writeproof[0,2,3][1,3,2]{A13, A14}
&\writeproof[0,2,3][1,3,2]{A13, A14}
&\writeproof[0,2,3][1,3,2]{A13, A14}
&\writeproof[0,2,3][1,3,2]{A13, A14}\\[0ex]\hline
\drawcaseleft{(1)--(2)}{(2)--(3)}{(3)--(0)}
&\writeproof*[0,5,4][5,4,1]{A0, A6}{A14}
&\writeproof*[0,5,4][5,4,1]{A7, A6}{A14}
&\writeproof[0,3,2][1,2,3]{A15, A16}
&\writeproof[0,3,2][1,2,3]{A15, A16}
&\writeproof[0,3,2][1,2,3]{A15, A16}
&\writeproof[0,3,2][1,2,3]{A15, A16}\\[0ex]\hline
\drawcaseleft{(1)--(2)}{(1)--(3)}{(3)--(0)}
&\writeproof[2,1,3]{A2, A7}{A3, A19}
&\writeproof[2,1,3]{A2, A4}{A20, A10}
&\writeproof*[0,2,3][4,1,5]{A11, A6}{A7}
&\writeproof*[0,2,3][0,5,1][4,0,5]{A8, A6}
&\writeproof*[0,2,3][4,1,5]{A8, A9}{A12}
&\writeproof*[4,0,5]{A9,A10}\\[0ex]\hline
\drawcaseleft{(0)--(2)}{(0)--(3)}{(1)--(3)}
&\writeproof[3,0,2]{A1, A2}
&\writeproof[3,0,2]{A1, A2}
&\writeproof[3,0,2]{A1, A2}
&\writeproof[3,0,2]{A1, A2}
&\writeproof[3,0,2]{A1, A2}
&\writeproof[3,0,2]{A1, A2}\\[0ex]\hline
\drawcaseleft{(1)--(2)}{(1)--(3)}{(2)--(0)}
&\writeproof{A5, A8}{A12, A18}
&\writeproof{A5, A6}{A10, A20}
&\writeproof{A1, A13}{A14}
&\writeproof{A1, A13}{A14}
&\writeproof{A1, A13}{A14}
&\writeproof{A1, A13}{A14}\\[0ex]\hline
\drawcaseleft{(0)--(2)}{(0)--(3)}{(1)--(2)}
&\writeproof{A5, A8}{A12, A18}
&\writeproof{A5, A6}{A10, A20}
&\writeproof*[1,3,2]{A0, A7}{A17}
&\writeproof*[1,3,2]{A0, A7}{A17}
&\writeproof*{A3, A4}{A9}
&\writeproof*{A3, A4}{A9}\\[0ex]\hline
\end{tabular}
\end{center}
\caption{Proofs for all $36$ possibilities for $T_1$.}
\label{prooftable}
\end{figure}

\begin{lemma} \label{k4plusk4}
The graph $K_4 +_e K_4$ is an excluded minor for $f_{\infty}(G) \leq 2$. 
\end{lemma}

\begin{proof}
By the previous lemma, $f_\infty(K_4 +_e K_4) = 3$, so it suffices to show that every proper minor $H$ of $K_4 +_e K_4$ satisfies $f_\infty(H) \leq 2$. 
Contracting an edge of $K_4 +_e K_4$ yields a $5$-vertex graph which is not $3$-connected. 
In particular, the latter graph does not have $W_4$ as a minor. 
We are done in this case, since by Lemma \ref{w4excluded}, $W_4$ is the only excluded minor for $f_{\infty}(G) \leq 2$ among graphs on at most $5$ vertices.

Deleting an edge from $K_4 +_e K_4$ creates a degree-$2$ vertex, which we can suppress by either Lemma \ref{subdivision} or Lemma \ref{degree2}. 
We then conclude as above, since  the resulting $5$-vertex graph is not $3$-connected and thus does not contain a $W_4$ minor. 
\end{proof}

\section{Proof of the Main Theorem} \label{sec:main}

The {\em wheel} on $n+1$ vertices, denoted by $W_n$, is the  graph obtained by adding a universal vertex to an $n$-cycle.
If $G$ and $G'$ are graphs such that $G=G' \backslash e$, we say that $G'$ is obtained from 
$G$ by \emph{adding an edge}.  Let $v \in V(G)$ with $\degree_G(v) \geq 4$.  By \emph{splitting $v$} we mean the operation of first deleting $v$, and then adding two new adjacent vertices $v_1$ and $v_2$, where
each neighbour of $v$ in $G$ is adjacent to exactly one of $v_1$ and $v_2$, and $v_1$ and $v_2$ have degree at least three in the new graph.

We require the following classic theorem of Tutte \cite{tutte}.

\begin{theorem}\label{thm:tutte}
(Tutte's wheel theorem)
Every 3-connected graph is obtained from a wheel by adding edges and splitting vertices.
\end{theorem}

The following characterization of graphs without a $W_4$ minor is well known. For the convenience of the reader, we give a quick proof via Theorem \ref{thm:tutte}.

\begin{theorem} \label{nowheel}
The only $3$-connected graph with no $W_4$ minor is $K_4$. 
\end{theorem}

\begin{proof}
 Let $G$ be a $3$-connected graph with no $W_4$ minor.  
By Tutte's wheel theorem, $G$ is obtained from some $W_n$ by adding edges and splitting vertices.  Since $G$ has no $W_4$ minor, we must have $n=3$. 
If $G \neq W_3$, then we get a contradiction, since there is no way to add an edge to $W_3$ and stay simple, and there is no way to split a vertex ($W_3$ is cubic).  Thus, $G=W_3=K_4$, as required. 
\end{proof}

We also need the following two technical lemmas.

\begin{lemma} \label{rootedk4}
Let $G$ be a $2$-connected graph and $u$ and $v$ be distinct vertices of $G$.  If $G$ has a $K_4$ minor, then $G$ has a $K_4$ minor $K$
where $u$ and $v$ are contracted to distinct vertices of $K$.  
\end{lemma}

\begin{proof}
Let $u$ and $v$ be distinct vertices of $G$.  Since $G$ has a $K_4$ minor and $K_4$ is cubic, $G$ also has a subgraph $H$ which is a subdivision of $K_4$.  By Menger's theorem, there are two disjoint paths from $\{u,v\}$ to $V(H)$.  By contracting these paths onto $V(H)$, we may assume that $u,v \in V(H)$.  But now in $H$ we can contract $u$ and $v$ onto distinct branch vertices of $K_4$.    
\end{proof}

We let $K_4-e$ denote the graph obtained from $K_4$ by removing an edge $e$.

\begin{lemma} \label{k4-e}
Let $G$ be a  $2$-connected graph with distinct vertices $u$ and $v$ such that $\degree (w) \geq 3$ for all $w \in V(G) \setminus \{u,v\}$.  Then $G$ has a $K_4-e$ minor where $u$ and $v$ are contracted to the endpoints  of $e$.  
\end{lemma}

\begin{proof}
Note that $G+uv$ has a $K_4$ minor since it has minimum degree $3$.  Thus, the result follows by applying Lemma \ref{rootedk4} to $G+uv$.  
\end{proof}

 Note that for all $p \in [1, \infty]$ and $m \in \mathbb{N}$, the property $f_p(G) \leq m$ is closed under $0$- and $1$-sums.  However, the graph $K_4 +_e K_4$ shows that the property $f_{\infty}(G) \leq 2$ is not closed under taking $2$-sums.  

We are now ready to prove our main result.

\begin{reptheorem}{main}
The excluded minors for $f_\infty (G) \leq 2$ are $W_4$  and $K_4 +_e K_4$. 
\end{reptheorem}

\begin{proof}
Let $G$ be a minor-minimal graph with $f_\infty(G) \geq 3$. By minimality and the preceding discussion, $G$ is $2$-connected.  By Lemmas \ref{degree2} and  \ref{subdivision} we may assume that $G$ has minimum degree~$3$.  By Lemmas \ref{w4excluded} and \ref{k4plusk4} we may assume that $G$ does not have a $W_4$ or $K_4 +_e K_4$ minor.  If $G$ is $3$-connected, then by Theorem  \ref{nowheel}, $G=K_4$, which is a contradiction since $f_\infty (K_4)=2$.  Thus, $G=G_1 +_f G_2$ or $G=G_1 \oplus_f G_2$ for some graphs $G_1$ and $G_2$ with $f:=ab \in E(G_1) \cap E(G_2)$ and $|E(G_1)|, |E(G_2)| > 1$.
Since $f \in E(G_1) \cap E(G_2)$ and $G$ is 2-connected it follows that $G_1$ and $G_2$ are both $2$-connected. 
By Lemma \ref{k4-e}, $G_1$ has a $K_4-e$ minor where $a$ and $b$ are contracted to the endpoints of $e$ and $G_2$ has a $K_4-e$ minor where $a$ and $b$ are contracted to the endpoints of $e$.  
Combining these  two minors we get a $K_4 +_f K_4$ minor in $G$, which is a contradiction.  
\end{proof}

Finally, we prove Corollary~\ref{cor:main}. 

\begin{repcorollary}{cor:main}
The excluded minors for $f_1 (G) \leq 2$ are $W_4$  and $K_4 +_e K_4$. 
\end{repcorollary}

\begin{proof}
Note that the map $\phi: \mathbb{R}^2 \to \mathbb{R}^2$ given by $(x,y) \to (\frac{x-y}{2},\frac{x+y}{2})$ is an 
isometry between the metric spaces $\ell_{\infty}^2$ and $\ell_{1}^2$. 
Thus for every graph $G$ and distance function $d$ on $G$, $(G,d)$ is realizable in $\ell_{\infty}^2$ if and only if it is realizable in $\ell_{1}^2$. Therefore, $f_{\infty}(G) \leqslant 2$ implies $f_1(G) \leqslant 2$.

Moreover, it follows from the equivalence between $\ell_1$-embeddability and membership in the cut cone~\cite{Avis77} and Seymour's linear description of the cut cone of $K_5$-minor free graphs~\cite{Seymour81} that \emph{every} distance function $d$ on a graph $G$ can be realized in some $\ell_{1}^m$ if $G$ is $K_5$-minor free. Hence for all $K_5$-minor free graphs $G$, we have $f_{\infty}(G) \leqslant 2$ if and only if $f_{1}(G) \leqslant 2$.

We claim that in fact $f_{\infty}(G) \leqslant 2$ if and only if $f_{1}(G) \leqslant 2$ for \emph{all} graphs $G$. Indeed, otherwise there would exist a graph $G$ such that $f_{\infty}(G) > 2$ and $f_{1}(G) \leq 2$. Then $G$ would have a $K_5$ minor, and thus 
$f_{1}(G) \geqslant f_{1}(K_5) \geqslant 3$ (the last inequality is proved in~\cite{W86}), a contradiction. The result follows.
\end{proof}

\section{The example and some open problems} \label{sec:openproblems}

A \emph{tree-decomposition} of a graph $G$ is a pair $(T, \mathcal{B})$
where $T$ is a tree and $\mathcal{B}:=\{B_t \mid t \in V(T)\}$ is a collection of subsets of vertices of $G$ satisfying:
\begin{itemize}
\item $G= \bigcup_{t \in V(T)} G[B_t]$, and
\item for each $v \in V(G)$, the set of all $w \in V(T)$ such that $v \in B_w$ induces a connected subtree of $T$.  
\end{itemize}
The \emph{width} of $(T, \mathcal{B})$ is $\max \{|B_t|-1 \mid t \in V(T)\}$.  The \emph{tree-width} of $G$ is the minimum width taken over all tree-decompositions of $G$. The \emph{path-width} of $G$ is defined analogously, except we insist that $T$ is a path instead of an arbitrary tree.  

Fix any tree $T$ with at least two vertices. Let $V^+ := \{v^+ \mid v \in V(T)\}$ and $V^- := \{v^- \mid v \in V(T)\}$ be two disjoint copies of $V(T)$. We construct a planar graph $T \circ K_4$ from $T$ by replacing each vertex $v$ of $T$ by a pair of vertices $v^+, v^-$ in $T \circ K_4$ and each edge $vw$ of $T$ by the $4$-clique $\{v^+, v^-, w^+, w^-\}$ in $T \circ K_4$.  Formally, $V(T \circ K_4) = \{v^+ \mid v \in V(T)\} \cup \{v^- \mid v \in V(T)\}$ and $E(T \circ K_4) = \{v^+ v^- \mid v \in V(T)\} \cup \{v^+w^-, v^+w^+, v^-w^+, v^-w^-  \mid vw \in E(T)\}$. We now prove the following strengthened form of Theorem~\ref{thm:example}.

\begin{reptheorem}{thm:example}
For every tree $T$ with at least two vertices, $T \circ K_4$ is planar with tree-width $3$ and $f_\infty(T \circ K_4) \geq |V(T)|$. 
\end{reptheorem}

\begin{proof}
Clearly, $T \circ K_4$ is planar since $K_4$ is planar and planarity is closed under taking $2$-sums.  It is also easy to see that $T \circ K_4$ has tree-width $3$.  For the last part, we order the edges of $T$ arbitrarily, and define a function $d : E(T \circ K_4) \to \R_{\geq 0}$ by letting $d_{v^+v^-} := 1$ for $v \in V(T)$, and $d_{v^+w^+} = d_{v^-w^-} := 2^{-i}$, $d_{v^+w^-} = d_{v^-w^+} := 1 - 2^{-i}$ for the $i$th edge $vw \in E(T)$.

\begin{claim}
The function $d : E(T \circ K_4) \to \R_{\geq 0}$ is a distance function on $T \circ K_4$.
\end{claim}

\begin{proof}
We have to check that $d(P) \geqslant d_e$ for all edges $e$ and all paths $P$ between the endpoints of $e$, where $d(P) := \sum_{f \in E(P)} d_f$. Clearly, the inequality is satisfied if $P$ contains the edge $e$. Similarly, if $P$ contains the edge $v^+v^-$ for some $v \in V(T)$ then $d(P) \geqslant d_{v^+v^-} = 1 \geqslant d_e$. Thus we may assume that $P$ is a path in $T \circ K_4 - (\{e\} \cup \{v^+v^- \mid v \in V(T)\})$. 

Every edge $f$ in the cut $\delta(V^+)$ has $d_f \geqslant 1-2^{-1} = \frac{1}{2}$. Hence, if $P$ contains at least two edges in the cut $\delta(V^+)$ then $d(P) \geqslant \frac{1}{2} + \frac{1}{2} = 1 \geqslant d_e$. So we may further assume that $P$ contains at most one edge in $\delta(V^+)$. 

Since $P$ does not contain the edge $e$, and $T \circ K_4[V^+]$ and $T \circ K_4[V^-]$ are both isomorphic to the tree $T$, the path $P$ cannot be completely contained in either of these induced subgraphs. Thus $P$ crosses $\delta(V^+)$ exactly once, and $e=u^+z^-$ for some $u^+$ and $z^-$. Let $f = v^+w^-$ denote the unique edge of $P$ in $\delta(V^+)$, where $vw \in E(T)$. Then $P$ consists of a path in $T \circ K_4[V^+]$ from $u^+$ to $v^+$, followed by the edge $v^+w^-$, followed by a path in $T \circ K_4[V^-]$ from $w^-$ to $z^-$. Thus $P$ contains $v^+ w^+$ or $v^- w^-$. Without loss of generality, $P$ contains $v^+ w^+$ and $d(P) \geqslant d_{v^+w^+} + d_{v^+w^-} = 2^{-i} + 1 - 2^{-i} = 1 \geq d_e$, where $i$ is the index of the edge $vw \in E(T)$.
\end{proof}

\begin{claim}
For all distinct $v, w \in V(T)$, no feasible set of $(T \circ K_4, d)$ can contain both $v^+v^-$ and $w^+w^-$.
\end{claim} 

\begin{proof} Let $e := v^+v^-$ and $f := w^+w^-$. There are only two possible feasible orientations of $\{e,f\}$ (up to reversing both edges).  Therefore, to prove the claim, it suffices to exhibit paths $P_1, P_2, Q_1, Q_2$ such that 
\begin{itemize}
    \item 
    $P_1$ has ends $v^+$ and $w^+$ and $P_2$ has ends $v^-$ and $w^-$, 
    \item
    $Q_1$ has ends $v^+$ and $w^-$ and $Q_2$ has ends $v^-$ and $w^+$,
    \item
    $d(P_1)+d(P_2) < d_e+d_f = 2$ and $d(Q_1)+d(Q_2) < d_e+d_f = 2$.
\end{itemize}

Consider the unique path $P = u_1 \cdots u_k$ in $T$ from $u_1 := v$ to $u_k := w$. 

We take $P_1 := u^+_1 \cdots u^+_k$ and $P_2 := u^-_1 \cdots u^-_k$. Then $d(P_1) = d(P_2)$ is a sum of distinct powers of two of the form $2^{-i}$ where $i \geqslant 1$ is an integer. Thus $d(P_1) = d(P_2) < \sum_{i=1}^\infty 2^{-i} = 1$ and in particular $d(P_1) + d(P_2) < 1 + 1 = 2$.

Pick $j$ in $\{1,\ldots,k-1\}$ such that in the ordering of $E(T)$, $u_ju_{j+1} \in E(T)$ is minimum. We take $Q_1 := u^+_1 \cdots u^+_j u^-_{j+1} \cdots u^-_k$ and $Q_2 := u^-_1 \cdots u^-_j u^+_{j+1} \cdots u^+_k$. Then 
$$
d(Q_1) = d(Q_2) < 1-2^{-i} + \sum_{\ell = i+1}^\infty 2^{-\ell} = 1 - 2^{-i} + 2^{-i} = 1\,.
$$
Thus $d(Q_1) + d(Q_2) < 1 + 1 = 2$, as required.
\end{proof}
Any realization of $(T\circ K_4,d)$ into $\ell_{\infty}^m$ implies a partition of the edges of $T\circ K_4$ into $m$ feasible sets. By the previous claim, no two of the edges of the form $v^+v^-$, where $v \in V(T)$ can be put in the same feasible set. Thus we have $f_{\infty}(T\circ K_4) \geqslant |V(T)|$.
\end{proof}

Note that by a classic result of Nash-Williams \cite{Nash-Williams64}, every planar graph can be partitioned into three forests.  Thus, Theorem~\ref{thm:example} shows that $f_\infty(G) - \Upsilon(G)$ can be arbitrarily large. Furthermore, by taking $T$ to be a path or a star in Theorem \ref{thm:example}, we see that $f_\infty$ is not bounded as a function of path-width or as a function of diameter.  

As promised, we finish the paper with a couple of open problems.  One natural question is to try to extend Theorem~\ref{main} to higher dimensions.

\begin{question}
What are the excluded minors for $f_{\infty}(G) \leq 3$?  
\end{question}

Let $P_4$ be a path with four vertices and $S_3$ be a star with three leaves.  By Theorem~\ref{thm:example}, $f_\infty(P_4 \circ K_4) \geq 4$  and $f_\infty(S_3 \circ K_4) \geq 4$.  Thus, $P_4 \circ K_4$ and $S_3 \circ K_4$ each contain an excluded minor for $f_\infty(G) \leq 3$.

Finally, it is also interesting to ask how the excluded minors for $f_p(G) \leq k$ change for $p \in [1, \infty]$.  Let $\mathcal{G}$ be the set of all finite graphs and define $\ex: [1, \infty] \times \mathbb{N} \to 2^{\mathcal{G}}$ by letting $\ex(p,k)$ be the set of excluded minors for $f_p(G) \leq k$.
Fix $k$ and define $p_1 \equiv_k p_2$ if $\ex(p_1, k)=\ex(p_2, k)$.  Note that $\equiv_k$ is an equivalence relation on $[1, \infty]$.  It may be possible to prove something about the structure of the equivalence classes of $\equiv_k$ without knowing the function $\ex(p,k)$.  For example, by the graph minor theorem, there are only countably many minor-closed properties.  Thus, some equivalence class of $\equiv_k$ is necessarily uncountable. 

\begin{question}
If $C$ is an equivalence class of $\equiv_k$ such that $|C|$ is uncountable, does $C$ necessarily contain an interval? 
\end{question}

\medskip
\noindent {\bf Acknowledgements.} 
We thank the two anonymous referees for their helpful comments and in particular for pointing out a gap in the proof of Corollary~\ref{cor:main} in a previous version of the paper.  
S.~Fiorini and T.~Huynh are supported in part by ERC grant \emph{FOREFRONT} (grant agreement no.\ 615640) funded by the European Research Council under the EU's 7th Framework Programme (FP7/2007-2013). S.~Fiorini also ackowledges support from \emph{ARC} grant AUWB-2012-12/17-ULB2 \emph{COPHYMA} funded by the French community of Belgium. A.~Varvitsiotis is supported in part by the Singapore National Research Foundation under NRF RF Award No. NRF-NRFF2013-13.

\bibliography{minors}
\bibliographystyle{abbrv}

\end{document}